\documentclass[11pt]{amsart}
\usepackage{amsfonts}
\usepackage{amssymb}
\usepackage{graphicx,color}

\usepackage{a4wide}

\usepackage{amsmath}

\usepackage{amsthm}

\graphicspath{{./Figs/}}
\DeclareGraphicsExtensions{.pdf,.jpg,.png,.eps,.gif}
\title{\bf Energy Bounds for Codes in Polynomial Metric Spaces}
\date{\today}

\newtheorem{theorem}{Theorem}[section]
\newtheorem{lemma}[theorem]{Lemma}
\newtheorem{corollary}[theorem]{Corollary}

\theoremstyle{definition}
\newtheorem{defn}[theorem]{Definition}
\newtheorem{example}[theorem]{Example}
\newtheorem{remark}[theorem]{Remark}

\begin{document}

\author{P. G. Boyvalenkov}\thanks{$^1$The research of the first and fifth authors was supported, in part, by a Bulgarian NSF contract DN02/2-2016.}
\address{ Institute of Mathematics and Informatics, Bulgarian Academy of Sciences, \\
8 G Bonchev Str., 1113  Sofia, Bulgaria
and Faculty of Mathematics and Natural Sciences, South-Western University, Blagoevgrad, Bulgaria }
\email{peter@math.bas.bg}

\author{P. D. Dragnev}\thanks{$^2$The research of the second author was supported, in part, by a Simons Foundation grant no. 282207.}
\address{ Department of Mathematical Sciences, Purdue University, \\
Fort Wayne, IN 46805, USA }
\email{dragnevp@ipfw.edu}

\author{D. P. Hardin}\thanks{$^3$The research of the third and fourth authors was supported, in part,
by the U. S. National Science Foundation under grant DMS-1516400.}
\address{ Center for Constructive Approximation, Department of Mathematics, \\
Vanderbilt University, Nashville, TN 37240, USA }
\email{doug.hardin@vanderbilt.edu}

\author{E. B. Saff}
\address{ Center for Constructive Approximation, Department of Mathematics, \\
Vanderbilt University, Nashville, TN 37240, USA }
\email{edward.b.saff@vanderbilt.edu}

\author{M. M. Stoyanova}
\address{ Faculty of Mathematics and Informatics, Sofia University, \\
5 James Bourchier Blvd., 1164 Sofia, Bulgaria}
\email{stoyanova@fmi.uni-sofia.bg}

\maketitle

\begin{abstract}
In this article we present a unified treatment for obtaining bounds on the potential energy of codes in the general context of polynomial metric spaces (PM-spaces).  The lower bounds we derive via the linear programming (LP) techniques of Delsarte and Levenshtein are universally optimal in the sense that they apply to a broad
class of energy functionals and, in general, cannot be improved for the specific subspace.
Tests are presented for determining whether these universal lower bounds (ULB) can be improved on larger spaces. Our ULBs
are applicable on the Euclidean sphere, infinite projective spaces, as well as Hamming and Johnson spaces.
Asymptotic results for the ULB for the Euclidean spheres and the binary Hamming space are
derived for the case when the cardinality and dimension of the space
grow large in a related way. Our results emphasize the common features of
the Levenshtein's universal upper bounds for the cardinality of codes with given separation and our ULBs for energy.
We also introduce upper bounds for the energy of designs in PM-spaces and the energy of codes with given separation.
\end{abstract}



\section{Introduction}

We consider polynomial metric spaces (PM-spaces) which include compact connected two-point homogeneous spaces
(also called compact symmetric spaces of rank one) when infinite
\cite{CS,DGS1,God,Hel,KL,Lev3,Lev-chapter,Slo1,Vil,Wan} and P- and Q-polynomial association schemes when finite \cite{BIto,BCN,Del1,DL,God1,God,Leo,Leo1,Lev-chapter,Neu,Ter}.
We describe below the main features of general PM-spaces. More
detailed examples are given in Section 3.

Let  $({\mathcal  M},d)$ be a compact metric space with   finite diameter $\Delta$, and $\mu$ a
Borel probability measure on ${\mathcal  M}$.  Let ${\mathcal  L}_2({\mathcal  M},\mu)$ denote the
Hilbert space of functions  $f: \mathcal{M} \to \mathbb{C}$ such that
$$\int_{\mathcal{M}} |f(x)|^2 d \mu(x)<\infty, $$
with inner product
\begin{equation}
\label{inner-prod}
\langle u,v \rangle :=\int_{\mathcal{M}} u(x)\overline{v(x)} d \mu(x).
\end{equation}
A continuous (strictly) decreasing function
$\sigma : [0,\Delta] \to [-1,1]
$
 such that $\sigma(0)=1$
and $\sigma(\Delta)=-1$ is called  a {\em substitution} for ${\mathcal  M}$.

Then $({\mathcal  M},d)$ with measure $\mu$ and substitution  $\sigma$ is called a  {\em polynomial metric space (PM-space)}
if  there is a finite or countably infinite collection of mutually orthogonal finite-dimensional subspaces $V_i$, $i\in \mathcal{I}$,
where $\mathcal{I}$  is an index set consisting of consecutive nonnegative integers starting at 0, and a collection of  real polynomials $Q_i(t)$,
$i=0,1,\ldots$, of respective degrees $i$, such that
\[    {\mathcal  L}_2({\mathcal  M},\mu) =\bigoplus_{i\in \mathcal{I}} V_i, \]
and for all $x,y \in {\mathcal  M}$,
\begin{equation}
\label{zsf}
   Q_i(\sigma(d(x,y)))=\frac{1}{r_i} \sum_{j=1}^{r_i} v_{ij}(x)
          \overline{v_{ij}(y)},
\end{equation}
where   $r_i=\dim(V_i)$ and $\{ v_{ij}(x) : 1 \leq j \leq r_i \}$ is any
orthonormal basis of $V_i$.

\begin{remark}
We collect several remarks about the notion of PM-spaces.
 \begin{enumerate}
 \item The right-hand side of \eqref{zsf} is the kernel for the orthogonal projection onto $V_i$ and so is independent of the orthonormal basis  chosen for $V_i$.
 \item A metric space  $({\mathcal M},d)$ with measure $\mu$ is called {\em distance invariant}  if, for a metric ball $B(x,r)$ with center $x\in {\mathcal M}$ and radius $r\ge 0$, the quantity  $\mu(B(x,r))$ depends only on $r$ and not on $x$.  For a given metric space  there is at most one such probability measure $\mu$.   In the case that ${\mathcal M}$ is finite, the measure $\mu$ must be normalized counting measure and the metric space $({\mathcal M},d)$ must be {\em distance regular}; i.e., the number of points that are a distance $r$ from a given point $x$ is independent of $x$.

If $g:\mathcal{M}\to \mathcal{M}$ is an isometry, then it follows that the measure $\mu$ is invariant under the action of $g$; i.e., $\mu(g(E))=\mu(E)$ for any $\mu$-measurable set $E\subset {\mathcal M}$.  Since  $ Q_i(\sigma(d(gx,gy)))= Q_i(\sigma(d(x,y)))$, the invariance of $\mu$ implies that $v_{ij}\circ g$ is also an orthonormal basis for $V_i$ showing that the subspaces $V_i$ are invariant under the action of $g$.

 \item The metric space $({\mathcal M}, d)$ is called {\em distance-transitive} if, for any points $x_1, x_2, y_1, y_2\in {\mathcal M}$ such that $d(x_1,x_2)=d(y_1,y_2)$, there is an isometry $g$ such that $y_i=g(x_i)$ for $i=1,2$.  In this case, there is a (unique) distance invariant probability measure $\mu$ (also called a {\em Haar measure}) that is invariant under the full group of isometries on ${\mathcal M}$.  Infinite connected compact metric spaces that are distance-transitive are also called ``two-point homogeneous spaces'' and have been classified by Wang   \cite{Wan} to be the Euclidean spheres
$\mathbb{S}^{n-1}$ (see Section 3) and the projective spaces $\mathbb{F}P^{n-1}$ (Section 3),
where $\mathbb{F}$ is the field of the real numbers $\mathbb{R}$, the field of the complex numbers $\mathbb{C}$, the
(non-commutative) division ring of quaternions $\mathbb{H}$, or the (non-associative) algebra
of octonions $\mathbb{O}$ (the space $\mathbb{O}P^{n-1}$ exists for $n=2,3$
only) \cite{CS,DGS1,God,Hel,KL,Lev3,Lev-chapter,Slo1,Vil,Wan}.
It is shown in \cite[Section 3.2]{Lev-chapter} that all these spaces are  PM-spaces.
 \end{enumerate}
\end{remark}

The finite PM-spaces are (P- and Q)-polynomial association schemes
\cite{BIto,BCN,Del1,DL,God1,God,Leo,Leo1,Lev-chapter,Neu,Ter}.
Two of the most important examples are Hamming spaces $H(n,q)$
with the Hamming metric and Johnson spaces $J(n,w)$ (see Section 3).

The system $\{Q_i(t)\}$ corresponding to a PM-space forms a sequence of orthogonal polynomials.
The orthogonality is with respect to the measure
\begin{equation}
\label{measure-ortho-poly-pms}
\nu(t) :=1-\overline{\mu}(\sigma^{-1}(t)),
\end{equation}
where
\[ \overline{\mu}(r) :=\int_{\mathcal{M}} \mu(B(x,r)) d \mu(x) \]
characterizes the mean measure of a metric ball $B(x,r)$ of center $x$ and radius $r$ (see \cite[Section 2]{Lev-chapter}).
The properties of the system  $\{Q_i(t)\}$ are crucial for many important results in PM-spaces. Note that $Q_0(t) \equiv 1$
and $Q_i(1)=1$ for every $i$, which follows from \eqref{zsf}.

\begin{defn} We denote by $T(\mathcal{M}) \subseteq [-1,1]$ the image of $\sigma$; i.e., the
set of all possible values of the function $\sigma(d(x,y))$, $x,y \in \mathcal{M}$.
\end{defn}

\begin{defn}
A {\em code} $C \subset {\mathcal  M}$ is a non-empty finite set. The maximum value of
the function $\sigma(d(x,y))$ on distinct points of $C$ is denoted by $s(C)$; i.e.,
\[ s(C):=\max\{\sigma(d(x,y)): x,y \in C,\  x \neq y\} \in T(\mathcal{M}). \]
Furthermore, $s(C)=\sigma(d(C))$, where $d(C):=\min \{ d(x,y): x,y \in C, \ x \neq y \}$ is the minimum distance (separation) of $C$.
\end{defn}

\begin{defn} For given ${\mathcal  M}$ and $s \in [-1,1)$, the maximum possible cardinality among all codes $C \subset {\mathcal  M}$
of given $s(C)=s$ is denoted by $A({\mathcal  M},s)$.
\end{defn}

\begin{defn}
A code $C \subset {\mathcal  M}$ is called a $\tau$-\emph{design}, if and only if the equality
\begin{equation}
\label{des-def}
\sum_{x,y \in C} Q_i(\sigma (d(x,y)))=0
\end{equation}
holds true for every $i=1,\ldots,\tau$. The maximum $\tau=\tau(C)$ such that $C$ is a $\tau$-design, is called
the \emph{strength} of $C$.
\end{defn}

\begin{defn} For given
${\mathcal  M}$ and positive integer $\tau$, the minimum possible cardinality among all $\tau$-designs
in ${\mathcal  M}$ is denoted by $B({\mathcal  M},\tau)$.
\end{defn}

The problems for finding upper bounds for $A({\mathcal  M},s)$ and lower bounds for
$B({\mathcal  M},\tau)$ are strongly related (see \cite{Del1,DL,God1,God,Lev-chapter} for discussions).
We review this relationship, and, furthermore, we extend it by
another one of its features -- the problem of obtaining lower bounds for the energy of codes in PM-spaces.

\begin{defn}
For a code $C \subset {\mathcal  M}$ and for a given (extended real-valued) function
$h(t):[-1,1] \to (0,+\infty]$,  the {\em $h$-energy} of $C$ is defined by
\begin{equation} E_h({\mathcal  M},C):=
\frac{1}{|C|}\sum_{x, y \in C, x \neq y} h(\sigma(d(x,y))). \nonumber\end{equation}
\end{defn}

The problem of minimizing the $h$-energy \cite{BHS} provided the cardinality of $C$ is fixed is commonly arising in the study of PM-spaces.

\begin{defn}
For given ${\mathcal  M}$ and positive integer $M \geq 2$, the minimum possible
$h$-energy of a code $C \subset {\mathcal  M}$ of cardinality $M$ is denoted by $E_h({\mathcal  M},M)$; i.e.,
\begin{equation}
 E_h({\mathcal  M},M):=\inf\{E_h({\mathcal  M},C):|C|=M, \, C\subset {\mathcal  M}\}. \nonumber
 \end{equation}
\end{defn}

Although the theorems that will be presented in Section 2 hold for general potentials $h$ we will be
especially concerned with functions that are {\em absolutely monotone}
({\em strictly absolutely monotone}); that is, $h^{(i)}(t)\geq 0$, $i=0,1,\dots$ ($h^{(i)}(t)> 0$, $i=0,1,\dots$) for
all $t \in [-1,1]$.
For the case of finite PM-spaces we note that if a function $F$ is absolutely monotone on $[-1,1]$ in the continuous sense,
then its restriction to $T(\mathcal{M})$ will be absolutely monotone in the discrete sense since the discrete derivative
$\delta^k F(t)=F^k(t^\prime) (\xi)$ for some $\xi \in (t,t^\prime)$. Similarly, we consider corresponding polynomials (Krawtchouk
polynomials, Hahn polynomials, etc.) in the continuous variable. Our setting, while somewhat restrictive, allows for a unified definition,
proof, and investigation of universal (in sense of Levenshtein, see \cite{Lev-chapter}) bounds
(Theorems \ref{thm 7} and \ref{thm testfunctions}).  Furthermore, such a continuous setting
facilitates the asymptotic analysis of our bounds in finite antipodal PM-spaces (as we show in Section \ref{asymp-antipodal}).

\begin{defn}
\label{antipodal spaces}
A PM-space ${\mathcal  M}$ is called \emph{antipodal} if for every point $x \in {\mathcal
M}$ there exists a point $\overline{x} \in {\mathcal  M}$ such that  $\sigma(d(x,y))+\sigma(d(\overline{x},y))=0$ for
any point $y \in {\mathcal  M}$.
\end{defn}

The (antipodal to $x$) point $\overline{x}$ in Definition \ref{antipodal spaces}
is uniquely determined by the equality $d(x,\overline{x})=\Delta$;
i.e., $\sigma_{{\mathcal  M}}(d(x,\overline{x}))=-1$. In antipodal PM-spaces, the system $\{Q_i(t)\}$ is symmetric; i.e., $Q_i(t)=
(-1)^iQ_i(-t)$ for all $i$ and $t$. The Euclidean spheres and the
binary Hamming spaces are the most important examples of antipodal
spaces.

The paper is organized as follows. In Section 2 we introduce the techniques of linear programming
in their general form and the universal bounds on minimum/maximum cardinality of designs/codes
in PM-spaces. Then we explain the important $1/M$-quadrature rule. Section 3 is devoted
to an overview of the basic PM-spaces -- the Euclidean spheres, Hamming spaces, Johnson spaces and
infinite projective spaces. In Section 4 we prove the main result in the paper -- universal lower bounds
on the energy of codes in PM-spaces. The optimality of our bound is discussed in Section 5 where we prove
a necessary and sufficient condition for the existence of improvements by linear programming. Section 6 is
devoted to an investigation of the asymptotic behaviour of our bound in antipodal PM-spaces for a particular
asymptotic process. Section 7 discusses common features in bounding cardinalities and energies.
The final section considers bounds (lower and upper) for the energy of designs
in PM-spaces and upper bounds on the energy of codes with prescribed separation.

\section{Linear programming in PM-spaces}

\paragraph{General linear programming bounds for the size and energy of codes and designs in PM-spaces}

We first introduce some needed notation.
For any real polynomial $f(t)$ of degree $r$ we have the unique expansion
\[ f(t)= \sum_{i=0}^r f_i Q_i(t) \]
with well defined coefficients
\[ f_i := r_i \int_{-1}^1 f(t)Q_i(t) d\nu(t). \]
For finite PM-spaces all polynomials are considered modulo $\prod_{\alpha \in T(\mathcal{M})}(t-\alpha)$.

\begin{defn}
By $F_{\geq}$ (respectively $F_>$) we denote the set of all polynomials such that $f_i \geq 0$ (respectively,
$f_i>0$) for every $i$ (respectively for every $i \leq \deg(f)$).
\end{defn}

The next three theorems are folklore in estimating by linear programming (LP)
the quantities $A({\mathcal  M},s)$, $B({\mathcal  M},\tau)$
 and $E_h({\mathcal  M},M)$. The proofs easily follow by using the addition formula \eqref{zsf} in the $Q$-system expansion
of the ``$f$-energy" sum
\[ \sum_{x, y \in C} f(\sigma (d(x,y))) \]
(see \cite{Del1,DL,God,Lev-chapter,Yud,CK}), where $f$ is a real polynomial.

\begin{theorem}
\label{thm2}
{\rm (LP for maximal codes problem)}
Let $\mathcal{M}$ be fixed, $s \in [-1,1)$, and $f(t)$ be a real non-constant polynomial such that

{\rm (A1)} $f(t) \leq 0 $ for every $t \in T(\mathcal{M}) \cap [-1,s]$, and

{\rm (A2)} $f \in F_{\geq}$.

Then  $A(\mathcal{M},s) \leq f(1)/f_0$.
\end{theorem}

\begin{theorem}
\label{thm2-1}
{\rm (LP for minimum designs problem)}
Let $\mathcal{M}$ be fixed, $\tau$ be positive integer, and $f(t)$ be a real polynomial such that

{\rm (B1)} $f(t) \geq 0 $ for every $t \in T(\mathcal{M})$, and

{\rm (B2)} the coefficients in $f(t)= \sum_{i=0}^r
f_i Q_i(t)$ satisfy $f_0>0$, $f_i \geq 0$ for $i \geq \tau+1$.

Then  $B(\mathcal{M},\tau) \geq f(1)/f_0$.
\end{theorem}

\begin{theorem}
\label{thm 1}
{\rm (LP for minimum energy problem)}
Let $\mathcal{M}$ be fixed and $h$ be a function defined on $T(\mathcal{M})$.
If $f$ is a real polynomial such that {\rm (A2)} is satisfied and

{\rm (C1)} $f(t) \leq h(t)$ for every $t \in T(\mathcal{M})$,

\noindent
then $E_h(\mathcal{M},M) \geq M(f_0M-f(1))$ for every $M\ge 2$.
\end{theorem}

Different values for $s$ in Theorem \ref{thm2} (different $\tau$ in Theorem \ref{thm2-1}
and different $M$ in Theorem \ref{thm 1}, respectively) require different choices
of good polynomials. Three important examples of
universal bounds derived in PM-spaces by utilization of suitable polynomials are as follows:

(i) an upper bound on $A(\mathcal{M},s)$ obtained by Levenshtein \cite{Lev-chapter}, see the next subsection;

(ii) a lower bound on $B(\mathcal{M},\tau)$ obtained by different authors in different PM-spaces (see
 \cite{Lev-chapter}); in particular, by Rao \cite{Rao} for $\mathcal{M}=H(n,q)$ and by
 Delsarte-Goethals-Seidel \cite{DGS} for $\mathcal{M}=\mathbb{S}^{n-1}$;

(iii) a lower bound on $E_h(\mathcal{M},M)$ obtained by the authors for $\mathcal{M}=H(n,q)$ in \cite{BDHSS-DCC}
and for $\mathcal{M}=\mathbb{S}^{n-1}$ in \cite{BDHSS-CA}.

\paragraph{Universal bounds for the size of codes and designs in PM-spaces}

Adjacent (to the $Q$-system) polynomials $Q_i^{a,b}(t)$, $i=0,1,\ldots$,
are defined to satisfy the orthogonality condition
\begin{equation}\label{AdjOrthoRel} r_i^{a,b}c^{a,b}\int_{-1}^1 Q_i^{a,b}(t) Q_j^{a,b}(t) (1-t)^a(1+t)^b
                       d\nu(t)=\delta_{i,j} \end{equation}
and the normalizations $c^{a,b}\int_{-1}^1 (1-t)^a(1+t)^b d\nu(t)=1$ and $Q_i^{a,b}(1)=1$. (The constants $r_i^{a,b}$ are determined by \eqref{AdjOrthoRel} and these normalizations.)
The case \hbox{$a=b=0$} gives the $Q$-system. We denote by $t_i^{a,b}$ the largest zero of the
polynomial $Q_i^{a,b}(t)$, $i \geq 1$. Let $q_1:=1-\frac{1}{Q_1(-1)}$ for brevity.

A universal upper bound on $A({\mathcal  M},s)$ has been
derived in different PM-spaces by Levenshtein \cite{Lev2,Lev3} (see also \cite{Lev4,Lev-chapter}). This
bound can be stated in terms of the $Q$-system and the adjacent systems as follows: for every
$s \in \left[t_{k-1+\varepsilon}^{1,1-\varepsilon},t_k^{1,\varepsilon}\right]$ one has
\begin{equation}
\label{L-bound}
A({\mathcal  M},s) \leq
L_{\tau}({\mathcal  M},s) :=
  \left( q_{1}\right)^\varepsilon
  \left( 1-\frac{Q_{k-1}^{1,\varepsilon}(s)}{Q_k^{0,\varepsilon}(s)} \right)
  \sum_{i=0}^{k-1} r_i^{0,\varepsilon},
       \end{equation}
Here and below we use $\varepsilon \in \{0,1\}$ to distinguish between the odd and even cases of $\tau=2k-1+\varepsilon$.
In fact, the bound \eqref{L-bound} has been proved in PM-spaces where
some special conditions called `Krein conditions' and `strengthened Krein conditions' of
the $Q$-system are fulfilled (see \cite[Section 5]{Lev-chapter}). These
conditions, and therefore \eqref{L-bound}, hold true in all infinite PM-spaces,
as well as the Hamming and the Johnson spaces.

A universal lower bound on $B(\mathcal{M},\tau)$ (a counterpart of the Levenshtein bound)
that holds for different PM-spaces is as follows (see \cite{Del1,DL,Lev-chapter,Lev3,Dun,DGS1,Hog1,Rao}):
\begin{equation}
\label{des-bound}
B({\mathcal  M},\tau) \geq D({\mathcal  M},\tau) :=
\left(q_1\right)^\varepsilon \sum_{i=0}^{k-1+\varepsilon} r_i^{0,\varepsilon}.
\end{equation}
This bound can be obtained by using
the polynomials $(t+1)^{1-\varepsilon}\left(Q_{k-1+\varepsilon}^{1,1-\varepsilon}\right)^2$ of degree $\tau$
in Theorem \ref{thm2-1}. Designs that attain \eqref{des-bound} are called \emph{tight}.

The bounds \eqref{L-bound} and \eqref{des-bound} are strongly connected by the equalities
at both ends of the intervals $[t_{k-1+\varepsilon}^{1,1-\varepsilon},t_k^{1,\varepsilon}]$. We have
\begin{equation}\label{L-DGS1}
L_{2k-2+\varepsilon}(\mathcal{M},t_{k-1+\varepsilon}^{1,1-\varepsilon})=
L_{2k-1+\varepsilon}(\mathcal{M},t_{k-1+\varepsilon}^{1,1-\varepsilon}) = D(\mathcal{M},2k-1+\varepsilon).
\end{equation}
This connection is crucial for an appropriate choice of polynomials in our main result below.

\paragraph{$1/M$-quadrature rule of Levenshtein}

Let
\[ T_j^{a,b}(u,v) := \sum_{i=0}^j r_i^{a,b} Q_i^{a,b}(u) Q_i^{a,b}(v), \ \ a,b \in \{0,1\}. \]
The kernels $T_j^{a,b}$ have important connections with different classes of
adjacent polynomials via the Christoffel-Darboux formula.

The Levenshtein bound \eqref{L-bound} can be obtained by using in Theorem \ref{thm2} the polynomial
\begin{equation} \label{L-poly}
\begin{split}
f_{2k-1+\varepsilon}^{(s)}(t) & :=(t+1)^{\varepsilon}(t-s)\left(T_{k-1}^{1,\varepsilon}(t,s)\right)^2 \\
& = (t-\alpha_0)^{2-\varepsilon}(t-\alpha_1)^2\ldots(t-\alpha_{k-2+\varepsilon})^2(t-\alpha_{k-1+\varepsilon}),
\end{split}
\end{equation}
where $-1 \leq \alpha_0 <\alpha_1< \cdots < \alpha_{k-2+\varepsilon} < \alpha_{k-1+\varepsilon}=s$, $\varepsilon \in \{0,1\}$,
with $\varepsilon=1$ if and only if $\alpha_0=-1$.

By \cite[Theorem 5.39]{Lev-chapter} there exist positive weights
$\rho_i$, $i=0,1,\ldots,k-1+\varepsilon$, such that
for any $s\in [ t_{k-1+\varepsilon}^{1,1-\varepsilon} , t_k^{1,\varepsilon} ]$ and any real polynomial $f(t)$ of degree at most $2k-1+\varepsilon$ the equality
\begin{equation}
\label{f0-second}
f_0 = \frac{f(1)}{L_{2k-1+\varepsilon}({\mathcal  M},s)}+\sum_{i=0}^{k-1+\varepsilon} \rho_i
       f(\alpha_i)
\end{equation}
holds. We are interested in the special case when $L_{2k-1+\varepsilon}({\mathcal  M},s)=M$ and in the
following related parameters:

(i) degree $\tau=\tau(\mathcal{M},M):=2k-1+\varepsilon$, defined as the unique $\tau$ such that
$D(\mathcal{M},\tau)<M\leq D(\mathcal{M},\tau+1)$ (here \eqref{L-DGS1} is crucial);

(ii) nodes $\alpha_i$, which are the zeros of the polynomial $f_{2k-1+\varepsilon}^{(s)} (t)$ and also roots of
the equation $L_{2k-1+\varepsilon}({\mathcal  M},s)=M$;

(iii) positive weights $\rho_i$, which are important ingredients in the quadrature \eqref{f0-second}.

Under these circumstances, we call the formula \eqref{f0-second} an $1/M$-\emph{quadrature rule}.

\paragraph{Some further technicalities}

For a PM-space ${\mathcal  M}$ with a corresponding measure $\nu$ of orthogonality of the $Q$-system
we set ($m \geq 0$ is an integer)
\begin{equation}
\label{moments}
b_m:= \int_{-1}^1 t^m d \nu(t)
\end{equation}
(note that $b_0=1$ because of the normalization). For $f(t)=\sum_{i=0}^k a_it^i$ one has the expansion
$f(t)=\sum_{i=0}^k f_i Q_i(t) $ and it follows from \eqref{moments} that the coefficient $f_0$ and the
moments $b_m$ are connected by the formula
\begin{equation}
\label{f0-first}
f_0=\sum_{i=0}^k a_ib_i.
\end{equation}
The next lemma gives relations between the moments $b_j$ and the Levenshtein parameters $\rho_i$, $\alpha_i$.

\begin{lemma}
\label{L2.3} {\rm \cite{Lev3}} For every $m \in \{0,1,\ldots,2k-1+\varepsilon\}$,
\begin{equation}
\label{powersums}
\frac{1}{L_{2k-1+\varepsilon}({\mathcal  M},s)}+\sum_{i=0}^{k-1+\varepsilon} \rho_i \alpha_i^m =  b_m.
\end{equation}

\end{lemma}

\begin{proof} Set $f(t)=t^{m}$ in \eqref{f0-second} and use \eqref{moments}.
\end{proof}

In antipodal PM-spaces we have $b_i=0$ for $i$ odd. This simplifies the formulas \eqref{f0-first} as well as the equations \eqref{powersums}.

\begin{example}
For $\mathcal{M}=\mathbb{S}^{n-1}$, the unit sphere in $\mathbb{R}^n$, we have
\[ b_{2j}=\int_{-1}^1 t^{2j} d\mu(t)=\frac{(2j-1)!!}{n(n+2) \ldots (n+2j-2)}. \]
In particular, $b_0=1$, $b_2=1/n$, $b_4=3/n(n+2)$, etc.
\end{example}

\begin{example}
In the binary Hamming space $H(n,2)$ we get
\[ b_{2j}=\frac{1}{2^n} \sum_{i=0}^n \left( 1-\frac{2i}{n} \right)^{2j} {n \choose i}, \]
and, in particular, $b_0=1$, $b_2=1/n$, $b_4=(3n-2)/n^3$, and $b_6=(15n^2-30n+16)/n^5$.
\end{example}

\section{Review of the basic PM-spaces}

\paragraph{Euclidean spheres $\mathbb{S}^{n-1}$}

For $x=(x_1,x_2,\ldots,x_n),\ y=(y_1,y_2,\ldots,y_n) \in \mathbb{S}^{n-1}$ we
have the usual Euclidean distance
\[ d(x,y)=\left((x_1-y_1)^2+\cdots+(x_n-y_n)^2\right)^{1/2}, \]
and the usual inner product $\langle x,y \rangle=x_1y_1+x_2y_2+\cdots+x_ny_n$,
connected by
\[ \langle x,y \rangle =1-\frac{d^2(x,y)}{2}. \]
This justifies the substitution $\sigma(d)=1-d^2/2:[0,2] \to [-1,1]$.
The measure $\mu$ is the normalized Lebesgue measure on $\mathbb{S}^{n-1}$ (the normalized surface area). Therefore,
\[ \overline{\mu}(d)=\frac{\sigma_{n-1}(\varphi)}{\sigma_{n-1}}, \ \cos \varphi=1-\frac{d^2}{2}, \]
where $\sigma_{n-1}(\varphi)$ is the surface area of a spherical cap of angular radius $\varphi$ and $\sigma_{n-1}=2\sigma_{n-1}(\pi/2)$
is the surface area of $\mathbb{S}^{n-1}$.

The space $V_i$ consists of the homogeneous harmonic polynomials in $n$ variables of total degree $i$. It is well
known that
\[ r_i=\dim(V_i)= \frac{2i+n-2}{i+n-2} \cdot {i+n-2 \choose i}. \]
The measure \eqref{measure-ortho-poly-pms}
is $d\nu(t) =c_n(1-t^2)^{(n-3)/2} d\, t$,
where  $c_n := \Gamma(\frac{n}{2})/\sqrt{\pi}\Gamma(\frac{n-1}{2})$ is a normalizing constant,
and the $Q$-system consists of the (normalized by $Q_i(1)=1$) Gegenbauer polynomials \cite{AS,Sze}, which satisfy
the three term recurrence relation
\[ (i+n-2)Q_{i+1}(t)=(2i+n-2)tQ_i(t)-iQ_{i-1}(t), \]
with initial conditions $Q_0(t)=1$ and $Q_1(t)=t$.

The codes/designs on $\mathbb{S}^{n-1}$ are finite sets of points of the sphere and are naturally called spherical codes/designs.
The bound \eqref{des-bound} was obtained
by Delsarte, Goethals, and Seidel \cite{DGS} and states
\begin{equation}
\label{DGS-bound}
B(n,\tau) \geq D(n,\tau) := {n+k-1-\varepsilon \choose n-1}+{n+k-2 \choose n-1},
\end{equation}
where $\tau=2k-1+\varepsilon$, $\varepsilon \in \{0,1\}$.
The Levenshtein bound is given by
\[ L_\tau(n,s)={k+n-3+\varepsilon \choose n-2} \left[ \frac{2k+n-3+2\varepsilon}{n-1}
-\frac{(1+s)^{\varepsilon}\left(Q_{k-1+\varepsilon}(s)-Q_{k+\varepsilon}(s)\right)}
{(1-s)\left(Q_{k+\varepsilon}(s)+(\varepsilon-1)Q_{k+\varepsilon}(s)\right)} \right], \]
where $\varepsilon \in \{0,1\}$, and was obtained in 1979 \cite{Lev79} (see also \cite{Lev2}).

\paragraph{Hamming spaces}

Let $n \geq 2$ and $q \geq 2$ be positive integers. The $q$-ary Hamming space $H(n,q)$
consists of vectors $x=(x_1,x_2,\ldots,x_n)$, where \hbox{$x_i \in \{0,1,\ldots,q-1\}$,} and the distance
between $x,y \in H(n,q)$ is the Hamming distance; i.e., the number of coordinates in which $x$ and $y$ differ. Then $\Delta=n$ and
\[ \overline{\mu}(d)=\frac{1}{q^n} \sum_{i=0}^d (q-1)^i {n \choose i} \]
is the normalized volume of a ball of radius $d$. The standard substitution is
\[ \sigma(d)=1-\frac{2d}{n}. \]
Therefore $T(H(n,q))=\{t_\ell:=1-2\ell/n : \ell=0,1,\ldots,n\}$; i.e.,
$\sigma(\ell)=t_\ell$.

The Hamming analog of the spherical harmonics is as follows.
Let $V_0$ consist of the constant function $1$ and, for $i=1,\ldots, n$, let $V_i$ consist of the $r_i$ functions
\[
\begin{split}V_i=\{u(x): H(n,q) &\to \mathbb{C} \mid u(x)=\xi^{\alpha_1 x_{j_1}+\cdots+\alpha_i x_{j_i}},\\
&1 \leq j_1<\cdots<j_i \leq n, \ \alpha_1,\ldots,\alpha_i \in \{1,\ldots,q-1\}\},
\end{split}
\]
where $\xi$ is a (complex) primitive $q$-th root of unity.  Denoting and enumerating the functions in $V_i$ by $Y_{ij}$, $j=1,\ldots, r_i$,
one easily to verifies that $V:= \{Y_{ij}: 0\le i\le n,\,  1\le j\le r_i\}$ is an orthonormal system with respect to the
``inner" product $\langle u,v\rangle =q^{-n} \sum_{x \in {H} (n,q)} u(x) \overline{v(x)}$  (see \cite[Theorem 2.1]{Lev4}).
Therefore, $r_i=(q-1)^i {n \choose i}$ and the addition formula \eqref{zsf} relates the $Q$-system
and the orthonormal systems $V_i$, $i=0,\ldots, n$.

The above implies that the $Q$-system is defined by
\[ Q_i(t)=\frac{1}{r_i} K_i^{n,q} (n(1-t)/2), \]
where
\[ K_i^{n,q}(z):=\sum_{j=0}^i (q-1)^{i-j} {z \choose j}{n-z \choose i-j} \]
are the $q$-ary Krawtchouk polynomials \cite{Kraw,Sze}. The orthogonality is given by
\[ \frac{r_i}{q^n} \sum_{\ell=0}^n r_\ell Q_i(t_\ell)Q_j(t_\ell)=\delta_{ij}, \ i,j=0,1,\ldots,n.  \]

The binary space $H(n,2)$ is antipodal while the spaces $H(n,q)$ with $q \geq 3$ are clearly not antipodal.

The codes in $H(n,q)$ are known as {\em error-correcting codes} \cite{MWS,HandbookECC} since they are
capable of correcting $\lfloor(d-1)/2\rfloor$ errors if their minimum distance
is $d$. The $\tau$-designs are widely known as orthogonal arrays \cite{HSS}.
A $\tau$-design $C \subset H(n,q)$ of strength $\tau$ is a code $C \subset H(n,q)$ of cardinality $|C|=M= \lambda q^\tau$
such that the $M\times n$ matrix obtained from the
codewords of $C$ as rows has the following property: every $M \times \tau$ submatrix
contains every element of $H(\tau,q)$  exactly $\lambda=\frac{M}{q^{\tau}}$ times as rows
(the positive integer $\lambda$ is called index of $C$). The characterization of codes by their strength as designs was
initiated by Delsarte \cite{Del1}, where $\tau+1=d^\prime$ is the dual distance of the (linear)
code $C$ (see also \cite{DL,Lev4,Lev-chapter}).

The bound  \eqref{des-bound} in the Hamming spaces can be proved by combinatorial
 arguments and was obtained by Rao \cite{Rao} in 1947:
\begin{equation}
\label{R-bound}
B(n,\tau) \geq R(n,\tau) := q^{1-\varepsilon} \sum_{i=0}^{k-1+\varepsilon} {n-1+\varepsilon \choose i} (q-1)^i,
\end{equation}
$\tau = 2k-1+\varepsilon$, $\varepsilon \in \{0,1\}$. The Levenshtein bound for $A(H(n,q),s)=:A_q(n,s)$ can be written as
\[ A_q(n,s) \leq L_{2k-1+\varepsilon}(n,s) =
        q^{\varepsilon} \left(1 - \frac{Q_{k-1}^{1,\varepsilon}(s)}{Q_k^{0,\varepsilon}(s)}\right) \sum_{i=0}^{k-1} {n-\varepsilon \choose i}(q-1)^{i},
\]
$\varepsilon \in \{0,1\}$.

\paragraph{Johnson spaces}
Let $n \geq 2$ and $1 \leq w \leq \lfloor n/2 \rfloor$ be positive integers.
The Johnson space $J(n,w)$ consists of binary words of length $n$ having
exactly $w$ ones (in other words, having {\em weight} ${\rm wt}(x):=w$; so the Johnson space
$J(n,w)$ is the subset of the binary Hamming space $H(n,2)$ consisting of all binary vectors
of weight $w$).

The distance between $x,y \in J(n,w)$ is the half of the Hamming distance between $x$ and $y$ (equivalently, $d(x,y)=w-{\rm wt}(x*y)$, where
the vector $x*y$ has ones exactly in the places where $x$ and $y$ simultaneously have ones). Then $\Delta=w$ and
 \[ \overline{\mu}(d)=\frac{1}{{n \choose w}} \sum_{i=0}^d {w \choose i} {n-i \choose i} \]
is the normalized volume of a ball of radius $d$. The standard substitution is
\[ \sigma(d)=1-\frac{2d}{w}. \]
Therefore $T(J(n,w))=\{t_\ell:=1-2\ell/w : \ell=0,1,\ldots,w\}$; i.e.,
$\sigma(\ell)=t_\ell$. We also note that $J(n,w)$ is antipodal if and only if $n=2w$.

Furthermore, we have $r_i={n \choose i}-{n \choose i-1}$ and the $Q$-system is defined by
\[ Q_i=J_i (w(1-t)/2), \]
where
\[ J_i(z):=\sum_{j=0}^i (-1)^{j} \frac{{i \choose j}{n+1-i \choose j}}{{w \choose j}{n-w \choose j}} {z \choose j} \]
are the Hahn polynomials \cite{Cheb,Hahn}.  The orthogonality is given by
\[ \frac{r_i}{{n \choose w}} \sum_{\ell=0}^n {w \choose \ell} {n-w \choose \ell} Q_i(t_\ell)Q_j(t_\ell)=\delta_{ij}, \ i,j=0,1,\ldots,w.  \]


The codes in $J(n,w)$ are known as {\em constant-weight codes} \cite{MWS}. The $\tau$-designs in $J(n,w)$ are the classical
$S_\lambda(\tau,w,n)$ (defined as a set $C$ of $w$-subsets of an $n$-set such that each $\tau$-subset of the $n$-set belongs
exactly to $\lambda$ of the $w$-subsets from $C$, where $\lambda=|C|{w \choose \tau}/{n \choose \tau}$; the designs with $\lambda=1$
are known as {\em Steiner systems}).

As in the Hamming space, the bound  \eqref{des-bound} for $J(n,w)$ can be proved by combinatorial arguments.
It was obtained by Ray-Chaudhuri and Wilson \cite{RCW} in 1975 for odd $\tau$ and by Dunkl \cite{Dun} in 1979 for even $\tau$, and asserts that
\[ B(n,\tau) \geq \left(\frac{n}{w}\right)^{\varepsilon} {n-\varepsilon \choose k-1},\quad  \tau = 2k-1+\varepsilon,\,  \varepsilon \in \{0,1\}.\]
 The Levenshtein bound for $A(J(n,w),s):=A(n,s,w)$ can be written as
\[  A(n,s,w) \leq L_{2k-1+\varepsilon}(n,s) =
      \left(\frac{n}{w}\right)^\varepsilon \left(1 - \frac{Q_{k-1}^{1,\varepsilon}(s)}{Q_k^{0,\varepsilon}(s)}\right) {n-1 \choose k-1}, \; \varepsilon \in \{0,1\}.\]

\paragraph{Projective spaces $\mathbb{F}P^{n-1}$, $\mathbb{F}=\mathbb{R}, \mathbb{C}$, and $\mathbb{H}$}

We continue the description of the infinite PM-spaces   following  the discussion given in \cite{Lev-chapter} (see Examples 2.4 and 2.11). Let
\[ \mathbb{F}^n:=\{x=(x_1,x_2,\ldots,x_n): x_i \in \mathbb{F}\}, \]
where $\mathbb{F}=\mathbb{R}$ (the field of the real numbers), $\mathbb{C}$ (the field of the complex numbers), or
$\mathbb{H}$ (the associative noncommutative quaternionic algebra).

The basis in $\mathbb{H}$ is formed by the elements $1,i,j,k$ such that $i^2=j^2=k^2=-1$, $ij=-ji=k$, and any $u \in \mathbb{H}$
can be represented as $u=u_0+u_1i+u_2j+u_3k$ with $u_0,u_1,u_2,u_3 \in \mathbb{R}$. Thus
$\mathbb{R} \subset \mathbb{C} \subset \mathbb{H}$.
For any $u \in \mathbb{H}$ one defines its conjugate element $u^{*}:=u_0-u_1i-u_2j-u_3k$ and verifies that $uu^{*}=u_0^2+u_1^2+u_2^2+u_3^2$.

Now for any $\mathbb{F}$ described above we have the norm $|u|=\sqrt{uu^{*}}$ and the identities $(uv)^{*}=v^{*}u^{*}$, $|uv|=|u|\cdot|v|$. Then
for vectors $x=(x_1,\ldots,x_n), y=(y_1,\ldots,y_n) \in \mathbb{F}^n$ we define their inner product by
\[ \langle x,y \rangle :=x_1y_1^{*}+\cdots+x_ny_n^{*}. \]

Further, $x$ and $y$ are called equivalent if there exists $\lambda \in \mathbb{F}$, $\lambda \neq 0$, such that $x_i=\lambda y_i$ for
every $i=1,2,\ldots,n$. Now the elements of the projective space $\mathbb{F}P^{n-1}$ are defined to be the equivalence classes (called \emph{lines}) of the
non-zero vectors of $\mathbb{F}^n$. For any two lines $X,Y \in \mathbb{F}P^{n-1}$ the angle $\varphi(X,Y)$ between them is defined by
\[ \cos \varphi(X,Y) := \frac{|\langle x,y \rangle|}{\sqrt{|\langle x,x \rangle||\langle y,y \rangle|}} \in [0,1], \ x \in X, y \in Y. \]
Indeed, the right-hand side of the last equality does not depend on the particular choice of the vectors $x \in X$ and $y \in Y$ and uniquely defines
the angle $\varphi(X,Y) \in [0,\pi/2]$. Now the distance between the lines $X,Y \in \mathbb{F}P^{n-1}$ is defined by
\[ d(X,Y) :=\sqrt{1-\cos \varphi(X,Y)}=\sqrt{2} \sin \frac{\varphi(X,Y)}{2} \in [0,1]. \]
Thus the diameter $\Delta$ equals $1$.

The metric space $\mathbb{F}P^{n-1}$ is distance invariant and possesses a unique normalized invariant measure $\mu$. For any
$d=\sqrt{2} \sin\frac{\varphi}{2} \in [0,1]$ one has
\[ \overline{\mu}(d)=\frac{\Gamma(mn/2)}{\Gamma(m(n-1)/2)\Gamma(m/2)} \int_{\cos^2 \varphi}^1 (1-z)^{\frac{m(n-1)}{2}-1} z^{\frac{m}{2}-1} dz, \]
where $m$ is the dimension (1, 2, or 4) of $\mathbb{F}$ over $\mathbb{R}$.

The standard substitution is
\[ \sigma(d)=2(1-d^2)^2-1 \in [-1,1] \]
and the $Q$-system is given by the Jacobi polynomials
\[ Q_i(t)=P_i^{\left(\frac{m(n-1)}{2}-1,\frac{m}{2}-1\right)}(t). \]

\paragraph{A table with main parameters of PM-spaces}

We summarize in table format the main parameters of the PM-spaces from this section.

\medskip

\small{
\noindent
\begin{tabular}{|c|c|c|c|}
\hline
$\mathcal{M}$ & $d(x,y)$ & $\mu(x)$, $\overline{\mu}(d)$ & $\sigma(d)$   \\ \hline
$\mathbb{S}^{n-1}$ & $|x-y|$ & $\mu(x)=d\sigma_{n-1}(x)$ & $1-\frac{d^2}{2}$  \\ \hline
$H(n,q)$ & $|i:(x_i \neq y_i\}|$ & $\overline{\mu}(d)=\frac{1}{q^n} \sum_{i=0}^d (q-1)^i {n \choose i}$  & $1-\frac{2d}{n}$  \\ \hline
$J(n,w)$ & $w-wt(x*y)$ & $\overline{\mu}(d)=\frac{1}{{n \choose w}} \sum_{i=0}^d {w \choose i} {n-i \choose i}$ & $1-\frac{2d}{w}$  \\ \hline
$\mathbb{F}P^{n-1}$ & $\sqrt{2}\sin \frac{\varphi(X,Y)}{2}$ & $\overline{\mu}(d)=\gamma_{m,n}
\int_{\cos^2 \varphi}^1 (1-z)^{\frac{m(n-1)}{2}-1}z^{\frac{m}{2}-1} d z$ & $2(1-d^2)^2-1$    \\
\hline
\end{tabular}
}

\small{
\noindent
\begin{tabular}{|c|c|c|c|c|}
\hline
$\mathcal{M}$ & $\nu(t)=1-\overline{\mu}(\sigma^{-1}(t))$ & $Q_i(t)$ & $r_i$  \\ \hline
$\mathbb{S}^{n-1}$ & $\gamma_n (1-t^2)^\frac{n-3}{2}$ & $P_i^{(\frac{n-3}{2},\frac{n-3}{2})}(t)$ &  $\frac{2i+n-2}{i+n-2} {i+n-2 \choose i}$\\ \hline
$H(n,q)$ & $\frac{1}{q^n} \sum_{i=0}^d (q-1)^i {n \choose i} \delta_{-1+\frac{2i}{n}}$
& $\frac{1}{r_i} K_i^{n,q} (\frac{n(1-t)}{2})$ &  $(q-1)^i {n \choose i}$\\  \hline
$J(n,w)$  &  $\frac{1}{{n \choose w}} \sum_{i=0}^d {w \choose i} {n-i \choose i}\delta_{-1+\frac{2d}{w}}$
& $J_i (\frac{w(1-t)}{2})$ & ${n \choose i}-{n \choose i-1}$\\ \hline
$\mathbb{F}P^{n-1}$ & $ c_{m,n} (1-t)^{\frac{m(n-1)}{2}-1}(1+t)^{\frac{m}{2}-1} d \, t$ & $P_i^{(\alpha,\beta)}(t)$ &
$\frac{(2i+\alpha+\beta+1){i+\alpha+\beta \choose i}{i+\alpha \choose i}}{(\alpha+\beta+1){i+\beta \choose i}}$ \\
 & & $\alpha=\frac{m(n-1)}{2}-1$, & $\alpha=\frac{m(n-1)}{2}-1$,\\
 & & $\beta=\frac{m}{2}-1$ & $\beta=\frac{m}{2}-1$\\
\hline
\end{tabular}
}

\section{Universal lower bounds for $E_h(\mathcal{M},M)$ in PM-spaces}

In this section, we shall apply  Theorem \ref{thm 1} for a choice of polynomials that
is motivated by the relations \eqref{L-DGS1}.
Let the space $\mathcal{M}$ and the cardinality $M \geq 2$ be fixed. Then $M$ uniquely determines the interval
\[ \left(D(\mathcal{M},2k-1+\varepsilon),D(\mathcal{M},2k+\varepsilon)\right] ,\ \varepsilon\in \{ 0,1 \}, \]
in which it lies. In particular, the positive integer $\tau=2k-1+\varepsilon:=\tau(\mathcal{M},M)$ is uniquely determined.

Since the Levenshtein function $L_{2k-1+\varepsilon}(\mathcal{M},s)$ is continuous and strictly increasing
for $s \in [t_{k-1+\varepsilon}^{1,1-\varepsilon},t_k^{1,\varepsilon}]$ from
$D(\mathcal{M},2k-1+\varepsilon)$ to $D(\mathcal{M},2k+\varepsilon)$
(see \eqref{L-bound}-\eqref{L-DGS1}, cf. \cite[Lemma 5.37]{Lev-chapter}),
there exists a unique $s$ such that
\[ M=L_{2k-1+\varepsilon}(\mathcal{M},s). \] This $s$ uniquely determines the Levenshtein polynomial
\eqref{L-poly} (see \cite[Section 5]{Lev-chapter}). In particular, we get the roots of this polynomial
$\alpha_0,\alpha_1,\ldots,\alpha_{k-1+\varepsilon}=s$ that serve as nodes for the $1/M$-quadrature rule \eqref{f0-second} with
uniquely determined positive weights $\rho_0, \rho_1,\ldots,\rho_{k-1+\varepsilon}$.

As in the case with Levenshtein bounds, we shall appeal to the so-called Krein condition and
strengthened Krein condition (see the discussion in \cite{Lev-chapter} between Corollary 5.41 and Theorem 5.42).

\begin{defn} \cite{CGS,Lev-chapter}
\label{Krein-c}
The $Q$-system of $\mathcal{M}$ satisfies the \emph{Krein condition} if
\[ Q_i(t)Q_j(t) \in F_\geq \]
for every $i$ and $j$.
\end{defn}

The Krein condition is quite useful since it implies $fg \in F_\geq$ whenever $f \in F_\geq$ and
$g \in F_\geq$. In our considerations this is applied for combinations of polynomials $Q_i^{1,0}$ (note that
$Q_i^{1,0}(t)=T_i(t,1)/T_i(1,1) \in F_\geq$).

\begin{defn} \cite{Lev-chapter}
\label{s-Krein-c}
The $Q$-system of $\mathcal{M}$ satisfies the\emph{ strengthened Krein condition} if it satisfies the Krein condition
and, in addition,
\[ (t+1)Q_i^{1,1}(t)Q_j^{1,1}(t) \in F_\geq \]
for every $i$ and $j$.
\end{defn}

The strengthened Krein condition is satisfied in all major PM-spaces.
More precisely, the $Q$-system  of $\mathcal{M}$ satisfies the Krein condition for all $\mathcal{M}$ \cite[Corollary 3.13]{Lev-chapter}
and it satisfies the strengthened Krein condition for all infinite $\mathcal{M}$ under consideration \cite[Lemma 3.22]{Lev-chapter}
 as well as for all finite
decomposable\footnote{A finite PM-space $\mathcal{M}$ is called {\em decomposable} \cite{Lev3}
if there exist a positive integer $\ell \geq 2$ and metric subspaces $\mathcal{M}_i$, $i=1,\ldots,\ell$, of $\mathcal{M}$,
such that the following three conditions are satisfied: $\mathcal{M}=\cup_{i=1}^\ell \mathcal{M}_i$, each $\mathcal{M}_i$ is isometric to
the same $\widetilde{\mathcal{M}}$ with the same standard substitution $\sigma$, and for any $x,y \in \mathcal{M}$
the number of the subspaces $\mathcal{M}_i$ containing simultaneously $x$ and $y$ is equal to
$\frac{\ell |\widetilde{\mathcal{M}}|(\sigma(d(x,y)+1)}{2|\mathcal{M}|}$}
$\mathcal{M}$ (see \cite[Lemma 3.25]{Lev-chapter}); in particular, for the Hamming and
Johnson spaces \cite[Example 3.23]{Lev-chapter}. The strengthened Krein condition is also satisfied in all antipodal spaces.

\begin{defn}
For a finite nonempty multiset $T$ of points from $[-1,1)$ whose points have multiplicity two except possibly for the left endpoint $-1$,
we denote by $H_T(h)$ the Hermite interpolation polynomial that agrees with the potential function
$h$ at every point of $T$ to the order of its multiplicity.
\end{defn}

We now present our main result on lower bounds for energy.

\begin{theorem} {\rm (Universal lower bound (ULB) on energy)}
\label{thm 7}
Let $\mathcal{M}$ be a PM-space and $h$ be absolutely monotone on $[-1,1)$. If the $Q$-system
of $\mathcal{M}$ satisfies the strengthened Krein condition and
$M \in \left(D(\mathcal{M},2k-1+\varepsilon),D(\mathcal{M},2k+\varepsilon)\right]$ is a fixed positive integer, then
\begin{equation}
\label{ulb-main}
E_h(\mathcal{M},M) \geq
M^2\sum_{i=0}^{k-1+\varepsilon} \rho_i h(\alpha_i) .
\end{equation}
Moreover, the bound \eqref{ulb-main} cannot be improved by utilizing polynomials $f \in F_\geq$ of degree at most  $\tau=2k-1+\varepsilon$
satisfying $f(t) \leq h(t)$ for every $t \in [-1,1)$.
\end{theorem}

\begin{proof}
Our choice of the polynomial $f(t)$ to be applied in Theorem \ref{thm 1} is explained below
separately in the odd and even cases for $\tau=2k-1+\varepsilon$.

We first verify that condition (A2) of Theorem 2.2 is satisfied by using results from \cite{CK,CW} and the strengthened Krein condition.
Our proof is, in some sense, a generalization of the Levenshtein's proof of $f_{2k-1+\varepsilon}^{(s)}(t) \in \mathcal{F}_\geq$
(see \cite[Theorem 5.42]{Lev-chapter}).

We first deal with the case $\varepsilon=0$.
Consider the interpolation polynomial $f(t)=H_T(h)$,
where
\begin{eqnarray*}
T &=& (\alpha_0,\alpha_0,\alpha_1,\alpha_1,\ldots,\alpha_{k-1},\alpha_{k-1}) \\
  &=& (t_1,t_2,\ldots,t_{2k-1},t_{2k})
\end{eqnarray*}
is the ordered multiset (i.e., $t_{2i+1}=t_{2i+2}=\alpha_i$) of the touching points of $f$ and $h$. Then
(see Lemma 10 from \cite{CW}) the Newton formula
\[ f(t)=\sum_{i=0}^m h[t_1,\ldots,t_m] \prod_{j=1}^{m-1} (t-t_j)  \]
gives that $f$ is a nonnegative linear combination of the constant 1 and the partial products
\begin{equation}
\label{p-prod-odd}
\prod_{j=1}^m (t - t_j), \ m=1,2,\ldots,2k-1.
\end{equation}
We next apply a result from \cite{CK} together with the Krein condition to show that $f\in F_\geq$.

Indeed, it follows from \cite[Theorem 3.1]{CK} that all partial products
\[ (t-\alpha_0)(t-\alpha_1)\ldots(t-\alpha_{i}), \ i=0,1,\ldots,k-2 \]
expand in the system $\{Q_i^{1,0}(t)\}$ with nonnegative coefficients. Since $Q_i^{1,0}(t) \in F_>$,
it follows from the Krein condition that all partial products
\eqref{p-prod-odd} with $m \leq 2k-2$ belong to $F_>$. The last partial product (with $m=2k-1$)
is exactly the Levenshtein polynomial $f_{2k-1}^{(s)}(t)$ (see \eqref{L-poly}) which belongs to $F_>$ by \cite[Theorem 5.42]{Lev-chapter}.
Therefore $f=H_T(h) \in F_\geq$.

For $\varepsilon=1$ we need the strengthened Krein condition.
Now $f=H_T(h)$ is obtained from the multiset
\begin{eqnarray*}
T &=& (\alpha_0=-1,\alpha_1,\alpha_1,\ldots,\alpha_{k},\alpha_{k}) \\
  &=& (-1,t_1,t_2,\ldots,t_{2k-1},t_{2k})
\end{eqnarray*}
and is, therefore, a nonnegative linear combination of the constant 1 and the partial products
\begin{equation}
\label{p-prod-even}
(t+1)\prod_{j=1}^m (t - t_j), \ m=0,1,2,\ldots,2k-1
\end{equation}
(here the $m=0$ case is the constant 1).
Theorem 3.1 from \cite{CK} now implies that all partial products
\[ (t-\alpha_1)(t-\alpha_2)\cdots(t-\alpha_{i}), \ i=1,2,\ldots,k-1, \]
expand in the system $\{Q_i^{1,1}(t)\}$ with nonnegative coefficients.
Then all partial products from \eqref{p-prod-even} expand with
positive coefficients in $(t+1)Q_i^{1,1}(t)Q_j^{1,1}(t)$ and the
strengthened Krein condition completes the argument; i.e., $f=H_T(h) \in F_\geq$.
Again, the case $m=2k-1$ gives exactly the Levenshtein polynomial $f_{2k}^{(s)}(t)$ needed to complete the proof.

Next we observe that $\deg(f) \leq 2k-1+\varepsilon$ and it easily follows from the Rolle's Theorem
that $f(t) \leq h(t)$ for every $t \in [-1,1)$. Therefore the condition (A1) is also satisfied
and Theorem \ref{thm 1} can be applied.

We now calculate the ULB by using the $1/M$-quadrature rule \eqref{f0-second} and the
interpolation conditions for $f(t)$. Since the cardinality of the interpolation
multiset $T$ is $2k-\varepsilon+2$, we conclude that $\deg(f) \leq 2k-\varepsilon +1$. Therefore \eqref{f0-second}
can be applied for $f$ and we have
\[ f_0 = \frac{f(1)}{M}+ \sum_{i=0}^{k-1+\varepsilon} \rho_i f(\alpha_i) \iff f_0M-f(1)=M\sum_{i=0}^{k-1+\varepsilon} \rho_i f(\alpha_i). \]
Since $f(\alpha_i)=h(\alpha_i)$ from the interpolation, we obtain
\[ f_0M-f(1) = M \sum_{i=0}^{k-1+\varepsilon} \rho_i h(\alpha_i). \]

Finally we prove the optimality property of the our bound. Let $F(t)=\sum_{i=0}^r F_iQ_i(t)$
of degree $r \leq 2k-1+\varepsilon$ satisfy
$F(t) \leq h(t)$ for every $t \in [-1,1)$. Then we have from \eqref{f0-second} applied either for $f(t)$ and $F(t)$
\begin{eqnarray*}
f_0M-f(1) &=& M\sum_{i=0}^{k-1+\varepsilon} \rho_i f(\alpha_i)=M\sum_{i=0}^{k-1+\varepsilon} \rho_i h(\alpha_i) \\
&\geq& M\sum_{i=0}^{k-1+\varepsilon} \rho_i F(\alpha_i)=F_0M-F(1), \end{eqnarray*}
which means that $F(t)$ does not give a better bound than \eqref{ulb-main}. \end{proof}

\begin{remark}
Similar to the Levenshtein bound case (see \cite[Theorem 5.2]{Lev3} and \cite[Theorem 5.43]{Lev-chapter}), the `odd branch' of
our ULB (this with $\varepsilon=0$) is valid also for all cardinalities from the even interval $(D(\mathcal{M},2k),D(\mathcal{M},2k+1)]$.
The parameters arise in the same way since the Levenshtein function
$L_{2k-1}(\mathcal{M},s)$ increases from $D(\mathcal{M},2k)$ to infinity when $s \in [t_k^{1,0},t_k)$ (note that $t_k^{1,0}<t_k$)
and the proof is essentially the same. Therefore
in PM-spaces where the strengthened Krein condition is not valid (or not proved) we still have a valid ULB despite giving weaker values
than those which would come from the even branch if true.
\end{remark}

\section{On the global optimality of the ULB}

In this section we assume that the conditions of Theorem  \ref{thm 7} are satisfied and derive sufficient
conditions for the existence of polynomials that would improve the ULB from Section 4.
Observe that such improvements are only (but not necessarily) possible by
polynomials of degree larger than $\tau(\mathcal{M},M)=2k-1+\varepsilon$.

For a positive integer $j$ and $s \in [t_{k-1+\varepsilon}^{1,1-\varepsilon},t_k^{1,\varepsilon}]$
consider the following functions depending on the $Q$-system, $j$ and $s$:
\begin{equation}
\label{tst}
P_j({\mathcal  M},s):=\frac{1}{L_{2k-1+\varepsilon}({\mathcal  M},s)}+\sum_{i=0}^{k-1+\varepsilon} \rho_i
              Q_j(\alpha_i)
\end{equation}
(we recall the relation $L_{2k-1+\varepsilon}({\mathcal  M},s)=M$ when applicable).
Note that  $P_j({\mathcal  M},s)=0$ for every $1 \leq j
\leq 2k-1+\varepsilon$ and every $s \in [t_{k-1+\varepsilon}^{1,1-\varepsilon},t_k^{1,\varepsilon}]$ (this follows
immediately by setting $f(t)=Q_j(t)$ in \eqref{f0-second}). The functions \eqref{tst} were introduced in \cite{BDD}
for  $\mathcal{M}=\mathbb{S}^{n-1}$ and in \cite{BD1} for general  $\mathcal{M}$ with the purpose of investigating
the optimality of the Levenshtein bound (see  \cite[Theorem 5.47]{Lev-chapter}). We show below that the same functions
serve for investigating the optimality of our ULBs. The next theorem shows that the optimality of \eqref{ulb-main}
can be extended to a larger subspace.

\begin{theorem}\label{THM_subspace_improve}
Let $\mathcal{M}$ and $M \in \left(D(\mathcal{M},2k-1+\varepsilon),D(\mathcal{M},2k+\varepsilon)\right]$ be fixed and $h$ be absolutely monotone.
Let $I\subseteq \mathbb{N} \setminus \{1,\ldots,2k-1+\varepsilon\}$ be an index set. If $P_j({\mathcal  M},s)  \ge 0 $ for every $j\in I$, then
the bound \eqref{ulb-main} cannot be improved by any polynomial
$f \in \Lambda= {\rm span}\{Q_j \colon  j\in I \cup \{0,1,\ldots,2k-1+\varepsilon\}\} $
such that $f \in F_\geq$ and $f(t) \leq h(t)$ in $[-1,1]$.
\end{theorem}
\begin{proof}
Suppose that $f$ satisfies the conditions of the theorem and write
\begin{equation}
\label{n1}
f(t)= g(t)+\sum_{j\in I}  f_j Q_j(t),
\end{equation}
where $\deg(g) \leq 2k-1+\varepsilon$ and $f_j\ge 0$ for every $j\in I$.
Note that $f_0=g_0$.
Using the $1/M$-quadrature rule \eqref{f0-second} for $g(t)$, equation \eqref{tst} with $L_{2k-1+\varepsilon}({\mathcal  M},s)=M$,
and \eqref{n1} we obtain
\begin{eqnarray*}
Mf_0- f(1) &=& Mg_0 - f(1) = g(1)+M\sum_{i=0}^{k-1+\varepsilon} \rho_i g(\alpha_i) -f(1)\\
    &=& M\sum_{i=0}^{k-1+\varepsilon} \rho_i \left(f(\alpha_i)-\sum_{j\in I}  f_j Q_j(\alpha_i)\right)-\sum_{j\in I}  f_j \\
    &=& M\sum_{i=0}^{k-1+\varepsilon} \rho_i  f(\alpha_i)-\sum_{j\in I}  f_j \left(1+ M\sum_{i=0}^{k-1+\varepsilon}\rho_i Q_j(\alpha_i) \right)\\
    &=& M\sum_{i=0}^{k-1+\varepsilon} \rho_i  f(\alpha_i)-M\sum_{j\in I}f_jP_j({\mathcal  M},s) \le
        M\sum_{i=0}^{k-1+\varepsilon} \rho_i  h(\alpha_i).
\end{eqnarray*}
For the last inequality, we used that $f(\alpha_i) \leq h(\alpha_i)$ for $i=0,1,\ldots,k-1+\varepsilon$, $f_j \geq 0$ for every $j \in I$,
and $P_j({\mathcal  M},s) \geq 0$ again for every $j\in I$. Therefore,
\[ M(Mf_0-f(1)) \leq  M^2\sum_{i=0}^{k-1+\varepsilon} \rho_i  h(\alpha_i) \]
and since $Mg_0-g(1)\geq Mf_0 - f(1)$, the polynomial $f$ does not improve the ULB.
 \end{proof}

Theorem \ref{THM_subspace_improve} can be applied in certain cases for proving that universal optimality cannot be proved by using only LP.
This requires careful investigation of properties of the $Q$-system. Examples for $\mathcal{M}=\mathbb{S}^{n-1}$ can be found in
\cite[Sections 4.3-4.4]{BDHSS-CA}.

\begin{theorem}
\label{thm testfunctions}
Let $\mathcal{M}$ and $M \in \left(D(\mathcal{M},2k-1+\varepsilon),D(\mathcal{M},2k+\varepsilon)\right]$ be fixed and $h$ be an absolutely monotone function.
The bound \eqref{ulb-main} cannot be improved by a
polynomial $f$ of degree at least $\tau+1=2k+\varepsilon$ such that $f \in F_\geq$ and $f(t) \leq h(t)$ in $[-1,1]$
if $P_j(\mathcal{M},s) \geq 0$ for every $j \geq \tau+1$. Furthermore, if
$h$ is a strictly absolutely monotone function and
$P_j(\mathcal{M},s)<0$ for
some $j \geq \tau+1$, then \eqref{ulb-main} can be improved by a polynomial
of degree exactly $j$.
\end{theorem}

\begin{proof} The first part follows from Theorem \ref{THM_subspace_improve} for $I=\{2k+\varepsilon,2k+\varepsilon+1,\ldots\}$. For the second part, we show that \eqref{ulb-main} can be improved by using in Theorem \ref{thm 1} a polynomial
\[ f(t):=\eta Q_j(t)+g(t), \]
where $\eta >0$ and $g(t)$ of degree at most $2k-1+\varepsilon$ is properly chosen.

We consider the new potential function
\[ \tilde{h}(t):= h(t)-\eta Q_j^{}(t), \]
where $\eta>0$ is small enough so that
$\tilde{h}^{(i)}(t) \geq 0$ on $[-1,1]$ for all $i=0,1,\dots,j+1$. Since $h$ is strictly absolutely monotone,
this choice of $\eta$ is possible and the function $\tilde{h}(t)$ is in fact absolutely monotone, because
$\tilde{h}^{(i)}(t) =h^{(i)}(t)>0$ for $i>j+1$.

Now choose the polynomial $g(t)$ to be the Hermite interpolant of $\tilde{h}$ at the nodes $\{ \alpha_i\}$
exactly as in Theorem \ref{thm 7}. As in Theorem \ref{thm 7} we deduce that $g$ satisfies
condition (A2) and, analogously, the condition (C1) (i.e., $g \leq \tilde{h}$), implying that
$f$ satisfies both (A2) and (C1) (i.e., $f \leq h$) of Theorem \ref{thm 1}.

We next prove that the bound given by $f(t)$ is better than \eqref{ulb-main}.
Multiplying by $\rho_i$ and summing up the interpolation equalities $g(\alpha_i)=\tilde{h}(\alpha_i)$ we get
\[ \sum_{i=0}^{k-1+\varepsilon} \rho_i g(\alpha_i)=
\sum_{i=0}^{k-1+\varepsilon} \rho_i h(\alpha_i)-\eta \sum_{i=0}^{k-1+\varepsilon} \rho_i Q_j^{}(\alpha_i). \]
Since from \eqref{f0-second}, $M\sum_{i=0}^{k-1+\varepsilon} \rho_i g(\alpha_i)=Mg_0-g(1)$, and from \eqref{tst},
$$M\sum_{i=0}^{k-1+\varepsilon} \rho_i Q_j^{}(\alpha_i)=MP_j(\mathcal{M},s)-1,$$
 we obtain
\[ Mg_0-g(1)=M\sum_{i=0}^{k-1+\varepsilon} \rho_i h(\alpha_i)+\eta -M\eta P_j(\mathcal{M},s), \]
which implies (note that $f_0=g_0$ and $f(1)=g(1)+\eta$)
\[ Mf_0-f(1)=M\sum_{i=0}^{k-1+\varepsilon} \rho_i h(\alpha_i)-M\eta P_j(\mathcal{M},s)>M\sum_{i=0}^{k-1+\varepsilon} \rho_i h(\alpha_i); \]
i.e., the polynomial $f(t)$ gives a better bound.
\end{proof}

\section{Some asymptotic properties of the ULBs for $\mathbb{S}^{n-1}$ and $H(n,2)$}
\label{asymp-antipodal}

In this section we consider the asymptotic behaviour of the bound \eqref{ulb-main} for the
main antipodal spaces $\mathcal{M}=\mathbb{S}^{n-1}$ and $H(n,2)$. We introduce the  notation
$E_h^n(M)$ for the minimal $M$-point $h$-energy on these spaces and write $D_n(\tau)$ for $D(\mathcal{M},\tau)$.

We consider sequence of codes of cardinalities $(M_n)$ satisfying
\[ M_n=L_\tau(\mathcal{M},s) \in (D_n(\tau),D_n(\tau+1)] \]
for fixed $\tau=2k-1+\varepsilon$, $\varepsilon \in \{0,1\}$, $n= 2, 3, \ldots$, such that
\begin{equation}
\label{asymp-1}
\lim_{n \to \infty} \frac{M_n}{n^{k-1+\varepsilon}}=\frac{2-\varepsilon}{(k-1+\varepsilon)!}+\delta,
\end{equation}
where $\delta \geq 0$ is a constant. In this case it can be deduced from the explicit formulas
\eqref{DGS-bound} and \eqref{R-bound} that the Delsarte-Goethals-Seidel bounds in $\mathbb{S}^{n-1}$ and the
Rao bounds in $H(n,2)$ both satisfy
\[ D(\mathcal{M},2k-1+\varepsilon)=\frac{(2-\varepsilon)n^{k-1+\varepsilon}}{(k-1+\varepsilon)!} +o(n^{k-1+\varepsilon}). \]

Asymptotics for the parameters $\alpha_i$ and $\rho_i$ under the conditions \eqref{asymp-1}
were obtained in \cite{Boumova-Danev} for $\mathbb{S}^{n-1}$ and in \cite{BDHSS-DCC} for $H(n,2)$.
We collect this information in the next lemma.  Set $\delta_k:=1+\delta(k-1)!$.

\begin{lemma}
\label{l1-alpha} {\rm \cite{Boumova-Danev,BDHSS-DCC}} For the asymptotic process described above we have:

{\rm a)} $\lim_{n \to \infty} \alpha_i =0$ for $i=1,2,\ldots,k-1+\varepsilon$;

{\rm b)} $\alpha_0=-1$ for $\varepsilon=1$ and $\lim_{n \to \infty} \alpha_0 = -1/\delta_k$ for $\varepsilon=0$;

{\rm c)} $\rho_0 M_n \in [0,1]$ for $\varepsilon=1$;

{\rm d)} $\lim_{n \to \infty} \rho_0 M_n =\delta_k^{2k-1}$ for $\varepsilon=0$.
\end{lemma}

We need in fact to consider $\alpha_i$, $i \geq 1$, more precisely than in Lemma \ref{l1-alpha}a).

\begin{lemma}
\label{l2-alpha} {\rm (\cite{BD} for a), \cite{Lev-chapter} for b) and c))} We have

{\rm a)}  $t_k^{1,\varepsilon} \geq |\alpha_{k-1+\varepsilon}|>|\alpha_{1+\varepsilon}|>
|\alpha_{k-2+\varepsilon}|>|\alpha_{2+\varepsilon}|>\cdots$;

{\rm b)}  if $\tau=2k$, then $c_1/\sqrt{n} \leq \alpha_1 \leq c_2/\sqrt{n}$ for some positive constants $c_1$ and $c_2$;

{\rm c)} $d_1/\sqrt{n} \leq t_k^{1,\varepsilon} \leq d_2/\sqrt{n}$ for some constants positive $d_1$ and $d_2$.
\end{lemma}

Finally, set
\[ R(t):=\sum_{j=0}^{2k-1+\varepsilon} \frac{h^{(j)}(0)}{j!} t^j. \]

\begin{theorem}
\label{asy-thm}

{\rm a)} If $\tau=2k-1$ and $M_n$ is as in \eqref{asymp-1}, then
\[  \liminf_{n \to \infty} M_n \left(\sum_{i=0}^{k-1} \rho_i h(\alpha_i)-\sum_{j=0}^{k-1} \frac{h^{(2j)}(0)}{(2j)!} \cdot b_{2j}\right)
= \delta_k^{2k-1}\left(h\left(-\frac{1}{\delta_k}\right)-R\left(-\frac{1}{\delta_k}\right)\right)-R(1). \]

{\rm b)} If $\tau=2k$, $M_n$ is as in \eqref{asymp-1} and $\lim_{n \to \infty} M_n \rho_0=\rho \in [0,1]$, then
\[  \liminf_{n \to \infty} M_n \left(\sum_{i=0}^{k} \rho_ih(\alpha_i)-\sum_{j=0}^{k} \frac{h^{(2j)}(0)}{(2j)!} \cdot b_{2j}\right)
= \rho \left(h(-1)+R(-1)\right)-R(1). \]
\end{theorem}

\begin{proof}
By the Taylor expansion formula.
\[ h(\alpha_i)=\sum_{m=0}^{2k-1+\varepsilon} \frac{h^{(m)}(0)}{m!} \alpha_i^m+ R_{2k+\varepsilon,i}=R(\alpha_i)+R_{2k+\varepsilon,i}, \]
where $R_{2k+\varepsilon,i}=\frac{h^{(2k+\varepsilon)}(\xi_i)}{(2k+\varepsilon)!} |\alpha_i|^{2k+\varepsilon}$, $|\xi_i| \in (0,|\alpha_i^{}|)$, $i=0,1,\ldots,k-1$. Thus
\begin{eqnarray*}
\sum_{i=0}^{k-1+\varepsilon} \rho_i h(\alpha_i) &=&
\sum_{i=0}^{k-1+\varepsilon} \rho_i \left(\sum_{m=0}^{2k-1+\varepsilon} \frac{h^{(m)}(0)}{m!} \alpha_i^m +R_{2k+\varepsilon,i}\right) \\
&=& \sum_{m=0}^{2k-1+\varepsilon}  \frac{h^{(m)}(0)}{m!} \sum_{i=0}^{k-1+\varepsilon} \rho_i \alpha_i^m +
\sum_{i=0}^{k-1+\varepsilon} \rho_i R_{2k+\varepsilon,i}.
\end{eqnarray*}
Using the identity
\[ \sum_{i=0}^{k-1+\varepsilon} \rho_i \alpha_i^m=b_m-\frac{1}{L_{2k-1+\varepsilon}(\mathcal{M},s)}=b_m-\frac{1}{M_n} \]
from Lemma \ref{L2.3} and multiplying by $M_n$, we obtain
\begin{eqnarray*}
M_n \left(\sum_{i=0}^{k-1+\varepsilon} \rho_i h(\alpha_i)-\sum_{j=0}^{k-1+\varepsilon} \frac{h^{(2j)}(0)}{(2j)!} \cdot b_{2j}\right)
&=& -  \sum_{j=0}^{2k-1+\varepsilon} \frac{h^{(j)}(0)}{(j)!}+ M_n \sum_{i=0}^{k-1+\varepsilon} \rho_i R_{2k+\varepsilon,i} \\
&=& -  R(1)+ M_n \sum_{i=0}^{k-1+\varepsilon} \rho_i R_{2k+\varepsilon,i}.
\end{eqnarray*}

It remains to investigate the remainder term $M_n \sum_{i=0}^{k-1+\varepsilon} \rho_i R_{2k+\varepsilon,i}$. We separate the odd and even cases.

a) Observe that the inequalities from Lemma \ref{l2-alpha} imply
\begin{equation}
\label{estimate-other-alphas}
0 \leq M_n \rho_i R_{2k,i}=O(n^{-1})
\end{equation}
for $i=1,2,\ldots,k-1$ since $M_n \sim n^{k-1}$ and $|\alpha_{i}| \leq c_2/\sqrt{n}$. For $i=0$ we use Lemma \ref{l1-alpha}d) and then b).

b) In the even case we have $\alpha_0=-1$ and
\begin{equation}
\label{estimate-other-alphas-2}
0 \leq M_n \rho_i R_{2k,i}=O(n^{-1/2})
\end{equation}
for $i=1,2,\ldots,k$ since $M_n \sim n^{k}$ and $|\alpha_{i}| \leq \max\{c_2/\sqrt{n},d_2/\sqrt{n}\}$.
\end{proof}

The next corollary states the first two consequences of Theorem \ref{asy-thm}.

\begin{corollary}
If $M_n$ is as in \eqref{asymp-1}, then
\begin{equation}
\label{l-bound-a1}
\liminf_{n  \to \infty} \frac{E_h^n(M_n)}{M_n^2} \geq h(0)
\end{equation}
and
\begin{equation}
\label{l-bound-a2}
\liminf_{n  \to \infty} \frac{E_h^n(M)-h(0)M_n^2}{M_n^2} \cdot n \geq \frac{h^{\prime\prime}(0)}{2}.
\end{equation}
\end{corollary}

\section{Common features in bounding maximal codes and minimum energy}

In this section we point out the features that are similar between the problems for estimating
$A({\mathcal  M},s)$ and $E_h({\mathcal  M},M)$.

The most striking property is the coincidence of the zeros of the Levenshtein
polynomials and the interpolation nodes for our polynomials. Moreover, in both cases the resulting
polynomials are optimal in a sense (see \cite{Sid,Lev3} for the Levenshtein polynomials) and their bounds
cannot be improved by using polynomials of the same or lower degree.
In particular, the Levenshtein polynomials and the ULB polynomials are simultaneously good for
linear programming, as the positive definiteness of the Levenshtein polynomials is an important
ingredient in the proof of the positive definiteness of our polynomials.

Second, the test functions giving necessary and sufficient conditions for existence of improvements of
the Levenshtein bounds by linear programming (see \cite{BD1}, \cite[Theorem 5.47]{Lev-chapter})
and our ULB (Theorem \ref{thm testfunctions}) coincide. Since the
target of negative test functions is the same as well, the investigation is, in fact,
identical for both problems. For example, in \cite{BDHSS-CA} and \cite{BDHSS-DCC}
we used directly the results from the investigations in the papers \cite{BDD} and \cite{BD1} in the cases
of Euclidean spheres and binary Hamming spaces, respectively. In particular, it follows that in any fixed dimension $n \geq 3$,
the suitable parameter sets for the so-called sharp configurations \cite{Ban2017,CK,Lev3} are finite in number.

Furthermore, the coincidence zeros-nodes continues in the next level linear programming
bounds. These are bounds obtained with higher (than $\tau(\mathcal{M},M)$) degree polynomials in cases where negative test functions exist.  Furthermore, next level test functions
can be defined and investigated analogously. We will develop the corresponding framework (called ``lifting" of the Levenshtein framework) in a future work.

Last, but not least, we point out that the coincidence phenomenon appears in the
problem for obtaining upper bounds for the  cardinality and lower bounds for the energy of codes
with inner products in prescribed subinterval $[\ell,s]$ of $[-1,1]$. In this case,
signed measures positive on $[-1,s]$, $[\ell,1]$, and $[\ell,s]$ that are positive definite up to certain degrees set the framework.
This is discussed in \cite{BDHSS-subintervals}.

\section{Energy bounds for designs and codes with given separation}

\paragraph{Lower and upper bounds for energy of designs}

We are also interested in bounds for the minimum and maximum possible potential energies
of designs in $\mathcal{M}$. Given a PM-space $\mathcal{M}$, strength $\tau$ and cardinality $M>D(\mathcal{M},\tau)$,
we denote by  $L_h(\mathcal{M},M,\tau)$ and  $U_h(\mathcal{M},M,\tau)$ the minimum and maximum, respectively,
of the $h$-energy of $M$-point $\tau$-designs in $\mathcal{M}$; that is,
\begin{equation}\label{LUdef}
\begin{split} L_h(\mathcal{M},M,\tau) := \inf \{E_h(\mathcal{M},C): |C| = M, C \subset \mathcal{M}  \mbox{ is a $\tau$-design}\},   \\
  U_h(\mathcal{M},M,\tau) := \sup \{E_h(\mathcal{M},C): |C| = M, C \subset \mathcal{M} \mbox{ is a $\tau$-design}\}.
  \end{split}
  \end{equation}
If there exist no $\tau$-designs with cardinality $M$ in $\mathcal{M}$, we set $L_h(\mathcal{M},M,\tau)=\infty$
and $U_h(\mathcal{M},M,\tau) =-\infty$ as is standard for the inf and sup of the empty set.

Let
\begin{equation}\label{unNtau} u(\mathcal{M},M,\tau):=\sup \{u(C):C \subset \mathcal{M} \mbox{ is a $\tau$-design}, |C|=M\},
\end{equation}
where $u(C):=\max \{\langle x,y \rangle : x,y \in C, x \neq y\}$, and
\begin{equation}\label{lnNtau}\ell(\mathcal{M},M,\tau):=\inf \{\ell(C):C \subset \mathcal{M} \mbox{ is a $\tau$-design}, |C|=M\},
\end{equation}
where $\ell(C):=\min \{\langle x,y \rangle : x,y \in C, x \neq y\}$. The quantities $u(\mathcal{M},M,\tau)$ and $\ell(\mathcal{M},M,\tau)$
are useful in the linear programming method since they provide information about the structure of the designs under consideration.

\begin{theorem} \label{Thm-lb-designs} \cite{BDHSS-DRNA}
 Let $\mathcal{M}$ be a PM-space, $\tau$ and $M \geq D(\mathcal{M},\tau)$ be positive integers
and $h:[-1,1]\to[0,+\infty]$.     Suppose
 $I$ is a subset of $[-1,1)$ and $f(t) = \sum_{i=0}^{\deg(f)} f_i Q_i(t)$ is a real
polynomial such that
\begin{enumerate}
\item[{\rm (D1)}] $f(t) \leq h(t)$ for $t\in I$, and

\item[{\rm (D2)}] $f_i \geq 0$ for $i \geq \tau+1$.
\end{enumerate}
If $C \subset \mathcal{M}$ is a $\tau$-design of $|C|=M$ points such that $\langle x,y\rangle \in I$ for distinct points
$x,y\in C$, then
\[ E_h(\mathcal{M},C) \geq M(f_0M-f(1)). \]
In particular, if $[\ell(\mathcal{M},M,\tau),u(\mathcal{M},M,\tau)] \subseteq I$, then
\[ L_h(\mathcal{M},M,\tau) \geq M(f_0M-f(1)). \]
\end{theorem}

The next theorem discusses upper energy bounds.

\begin{theorem}
\label{Thm-ub-designs} \cite{BDHSS-DRNA}  Let $\mathcal{M}$, $\tau$, $M$, and $h$  be as in Theorem~\ref{Thm-lb-designs}.
    Suppose
 $I$ is a subset of $[-1,1)$ and $g(t) = \sum_{i=0}^{\deg(g)} g_i P_i^{(n)}(t)$ is a real polynomial such that

\begin{enumerate}
\item[{\rm (E1)}] $g(t) \geq h(t)$ for  $ t \in I$, and

\item[{\rm (E2)}] $g_i \leq 0$ for $i \geq \tau+1$.
\end{enumerate}
If $C \subset \mathcal{M}$ is a $\tau$-design of $|C|=M$ points such that $\langle x,y\rangle \in I$ for distinct points
$x,y\in C$, then
\[ E_h(\mathcal{M},C) \leq M(g_0M-g(1)). \]
In particular, if $[\ell(\mathcal{M},M,\tau),u(\mathcal{M},M,\tau)] \subseteq I$, then
\[ U_h(\mathcal{M},M,\tau) \leq M(g_0M-g(1)). \]
\end{theorem}

Lower and upper bounds on the energy of spherical designs were derived and discussed by the authors in \cite{BDHSS-DRNA}.

\paragraph{Upper bounds for energy of codes of given separation}

Yet another kind of energy bound can be obtained by linear programming under the assumption of prescribed separation.
Given a PM-space $\mathcal{M}$, cardinality $M \geq 2$, separation inner product $s$, and $h$,
we denote by $G_h(\mathcal{M},M,s)$ the maximal possible $h$-energy of
$M$-point codes in $\mathcal{M}$ with prescribed maximal $s(C)=\sigma(d(C))=s$; that is,
\begin{equation}
\label{Gdef}
\begin{split} G_h(\mathcal{M},M,s):=\sup \{E_h(\mathcal{M},C):|C|=M, s(C)=s\}.
  \end{split}
  \end{equation}
If there exist no codes with cardinality $M$ and maximal inner product $s$ in $\mathcal{M}$,
we set $G_h(\mathcal{M},M,s)=-\infty$.

\begin{theorem}\label{thm 9}
Let $\mathcal{M}$ be a PM-space, $M \geq 2$ be a positive integer, $s \in [-1,1)$ be fixed,
and $h$ be a function defined on $T(\mathcal{M})$.
Let $f(t)= \sum_{i=0}^r f_i Q_i(t)$ be a real polynomial such that

{\rm (F1)} $f(t) \geq h(t)$ for every $t \in T(\mathcal{M}) \cap [-1,s)$;

{\rm (F2)} $f_i \leq 0$ for $i=1,\ldots,k$.

\noindent
Then $G_h(\mathcal{M},M,s) \leq M(f_0M-f(1))$.
\end{theorem}

{\bf Acknowledgement.} Research for this article was conducted while the authors were in residence at the Institute for Computational and Experimental Research in Mathematics in Providence, RI, during the "Point Configurations in Geometry, Physics and Computer Science" program supported by the National Science Foundation under Grant No. DMS-1439786.

\section*{References}

\bibliography{mybibfile}

\end{document}